\documentclass[a4paper,12pt]{amsart}
\usepackage[bbgreekl]{mathbbol}
\usepackage{amsmath,amssymb,amsthm}
\usepackage{amscd}
\usepackage{enumerate}
\usepackage{cite}
\usepackage{mathrsfs}
\usepackage{multirow}

\allowdisplaybreaks

\newcommand{\R}{\mathbb R}
\newcommand{\N}{\mathbb N}
\newcommand{\Z}{\mathbb Z}
\newcommand{\C}{\mathbb C}

\newcommand{\dd}{\textup{d}}
\newcommand{\dbar}{{\mathchar'26\mkern-11mu \dd}}
\newcommand{\cl}{{\textup{cl}}}
\newcommand{\ls}{{\ell_*}}

\newcommand{\ee}{\textup{e}}
\newcommand{\ii}{\textup{i}}

\newcommand{\cC}{{\mathscr C}}
\newcommand{\cD}{{\mathscr D}}
\newcommand{\cS}{{\mathscr S}}

\newcommand{\bE}{{\mathbb E}}
\newcommand{\bF}{{\mathbb F}}
\newcommand{\bL}{{\mathbb L}}
\newcommand{\bT}{{\mathbb T}}
\newcommand{\bs}{{\bbsigma}}
\newcommand{\bt}{{\bbtau}}

\newcommand{\tr}{\textup{tr}}

\numberwithin{equation}{section}

\theoremstyle{plain}
\newtheorem{thm}{Theorem}[section]
\newtheorem{lem}[thm]{Lemma}
\newtheorem{prop}[thm]{Proposition}

\theoremstyle{definition}
\newtheorem{defn}[thm]{Definition}
\newtheorem{cor}[thm]{Corollary}

\theoremstyle{remark}

\newtheorem*{rem}{Remark}
\newtheorem*{ex}{Example}

\title[Hyperbolic problems with symbolic structure]{A scheme for solving
  hyperbolic \\ problems with symbolic structure}

\author[Z.-P.~Ruan]{Zhuoping Ruan}
\address{Department of Mathematics, Nanjing University, Nanjing
  210093, China}
\email{zhuopingruan@nju.edu.cn}

\author[I.~Witt]{Ingo Witt}
\address{Mathematical Institute, University of G\"{o}ttingen,
  Bunsenstr.~3-5, D-37073 G\"{o}ttingen, Germany}
\email{iwitt@uni-math.gwdg.de}



\thanks{The first author was supported by the NSFC
  (No.~11771206). Part of the work was done while the second author
  was visiting Nanjing University.}

\begin{document}


\begin{abstract}
Hyperbolic problems can at times be solved employing symbolic
arguments. This is especially true for the construction of forward
(and backward) fundamental solutions. We formulate a corresponding
abstract scheme and illustrate its practicality by a number of
instructive examples.
\end{abstract}


\maketitle
\tableofcontents


\section{Introduction}\label{Sec1}

In this article, we present a scheme for solving hyperbolic problems
with symbolic structure. The objective is to promote the idea that
using symbolic arguments is also suitable for hyperbolic problems. In
fact, symbolic arguments are widely used for elliptic equations, but,
less known, can likewise be effectively applied in the hyperbolic
realm. The main difference between both is that one obtains a
parametrix (i.e., an inverse up to a remainder in the residual class)
in the elliptic case, whereas one actually accomplishes to get a
genuine inverse for hyperbolic problems.

In the framework we propose the conditions to be imposed are the
existence of a principal symbol map, the unique solvability of the
transport equations, asymptotic completeness, and well-posedness on
the level of the residual class. See below for a precise formulation
and a discussion of these conditions.

\medskip

Specifically, we are interested in the following situation: \ Let
$\bL\colon \bE^0\longrightarrow \widetilde{\mathbb E}^0$ be a linear
continuous operator between Fr\'echet spaces. Suppose that, for a
given $f\in \widetilde{\mathbb E}^0$, we want to solve the equation
\begin{equation}\label{equ}
    \mathbb L u = f
\end{equation}
for $u\in \bE^0$. Suppose, in addition, that Eq.~\eqref{equ} comes
with a \textit{symbolic structure\/}.  By this we mean the following:
\ There are four sequences $\left\{\bE^j\right\}_{j\in\N_0}$,
$\left\{\bF^j\right\}_{j\in\N_0}$,
$\bigl\{\widetilde{\bE}^j\bigr\}_{j\in\N_0}$,
$\bigl\{\widetilde\bF^j\bigr\}_{j\in\N_0}$ of Fr\'echet spaces, where
\[
  \bE^0\supseteq \bE^1\supseteq \bE^2\supseteq \dotsc \supseteq
  \bE^\infty, \qquad \widetilde \bE^0\supseteq
  \widetilde\bE^1\supseteq \widetilde\bE^2\supseteq \dotsc \supseteq
  \widetilde \bE^\infty,
\]
$\bE^\infty= \bigcap_j\nolimits \bE^j$, $\widetilde\bE^\infty= \bigcap_j
\widetilde\bE^j$, and
\[
  \bL \in \bigcap\nolimits_{j\in\N_0} \mathcal
  L\bigl(\bE^j,\widetilde\bE^j\bigr).
\]
Moreover, there are linear continuous operators $\bs^j \colon \bE^j
\to \bF^j$, $\widetilde\bs^j \colon \widetilde\bE^j \to
\widetilde\bF^j$, and $\bT^j\colon \bF^j\to \widetilde\bF^j$ such
that, for each $j\in \N_0$, the diagram
\begin{equation}\label{diagr}
\begin{CD}
  0 @>>> \bE^{j+1} @>>> \bE^j @>\bs^j>> \bF^j @>>> 0 \\
  @. @V{\bL}VV @V{\bL}VV @VV{\bT^j}V @. \\
  0 @>>> \widetilde\bE^{j+1} @>>> \widetilde\bE^j @>\widetilde\bs^j>>
  \widetilde\bF^j @>>> 0.
 \end{CD}    
\end{equation}
commutes.

\bigskip
We make the following assumptions:

\medskip
\noindent
\textbf{(I) Principal symbol maps.} \ The rows in \eqref{diagr} are
exact.

\medskip
\noindent
\textbf{(II) Transport equations.} \ The operators $\mathbb \bT^j
\colon \mathbb \bF^j \to \widetilde{\mathbb F}^j$ are bijective.

\medskip
\noindent
\textbf{(III) Asymptotic completeness.} \ Given a sequence $\{u_j\}$
with $u_j\in \bE^j$ for all $j\in\N_0$, there is a $u\in \bE^0$ such
that, for all $J\in \N_0$,
\begin{equation}\label{nun}
  u - \sum\nolimits_{j<J} u_j \in \bE^J.
\end{equation}

\medskip
\noindent
\textbf{(IV) Well-posedness.} \ Given $f\in \widetilde\bE^\infty$,
Eq.~\eqref{equ} possesses a unique solution $u\in \bE^\infty$.

\bigskip

The main theorem is the following one:

\begin{thm}\label{main}
Suppose that properties\/ \textup{(I) through (IV)} hold. Then\/
\textup{Eq.~\eqref{equ}} possesses a unique solution $u\in \bE^0$ for
any $f\in \widetilde\bE^0$.
\end{thm}

In applications, it can be fairly hard to provide a framework in which
all the conditions (I) through (IV) are met. This is because to
establish such a framework means nothing less than to reveal the
analytic content of the problem under consideration. We shall discuss
five illustrative examples in Section~\ref{sec3}. An even more
sophisticated example will appear in \cite{RW2021}.

\bigskip

Let us add a few remarks before we proceed: 

\medskip
\noindent
(a) \ The open mapping theorem implies that the operator $\bL\colon
\bE^0 \to \widetilde\bE^0$ has a continuous inverse $\bL^{-1}\colon
\widetilde\bE^0 \to \bE^0$ if conditions (I) through (IV) hold. The same is
true for $\bL\colon \bE^j \to \widetilde\bE^j$ and any
$j\in\N_0\cup\{\infty\}$.

\medskip
\noindent
(b) \ Well-posedness in the residual class, as expressed by (IV), is
typical of hyperbolic problems. For elliptic problems, one usually has
conditions (I) through (III) only. Under these conditions, given $f\in
\widetilde \bE^0$, one finds a $u\in \bE^0$ such $\bL u - f\in
\widetilde\bE^\infty$. This $u$ is unique modulo $\bE^\infty$.

\medskip
\noindent
(c) \ Condition (II) is in effect an \textit{ellipticity
  condition\/}. Indeed, in the elliptic case, the operator $\bT^j$ is
often just multiplication by the principal symbol of the operator
$\bL$. In the hyperbolic case, however, by reasons coming from
microlocal analysis the principal symbol of $\bL$ usually vanishes so
that the next-order symbol takes control. (This also relates to the
fact that one does not gain as much regularity by solving
Eq.~\eqref{equ} as in the elliptic case.) In physics, the $\bT^j$ are
usually dubbed transport operators.

\medskip
\noindent
(d) \ Property \eqref{nun} is commonly written as $u \sim \sum_j u_j$
in $\bE^0$ understanding that, in fact, $\bE^0$ stands for the
filtration $\{\bE^j\}$ of $\bE^0$. Notice that this property
determines $u$ uniquely modulo $\bE^\infty$.

\bigskip

When symbolic arguments are exploited in the context of hyperbolic
problems, then this is usually done when working in local
coordinates. (The problem treated in \cite[Sec.~5]{Dui1996}
constitutes a notable exception. It will be reviewed in
Section~\ref{Sec32}.) In such a situation one has symbolic components
to all orders and, accordingly, transport equations are also solved to
all orders sort of instantaneously. One of the fine points about the
scheme disclosed here is that one does not have to care about
lower-order symbols, the determination of them is implicit and runs in
the background. Of course, details about lower-order symbols might be
worked out if necessary for the respective application.

\bigskip

The paper is organized as follows: \ In Section~2, we prove
Theorem~\ref{main} and discuss a few implications of the
assumptions. Section~3 is then devoted to examples which, as we hope,
demonstrate the usefulness of this method. In an appendix, we describe
an abstract scheme for establishing asymptotic completeness and, for
the reader's convenience, also recall the notion of a conormal
distribution and the transmission property.

\medskip

\noindent
\textbf{Notation.}  We use standard notation from microlocal analysis,
see e.g.~\cite{Hoe1985a,Hoe1985b,Tay1981}.

\begin{itemize}
\item \textbf{(Pseudodifferential operators)}
  \ $\Psi_{\textup{cl}}^m(X)$ for $m\in\C$ is the space of classical
  pseudodifferential operators of order $m$ on a $\cC^\infty$~manifold
  $X$ and $\sigma_\psi^m(P)\in S^{(m)}(\dot T^*X)$ is the principal
  symbol of an operator $P\in \Psi_{\textup{cl}}^m(X)$.

\item \textbf{(Fourier integral operators)} \ For $\mu\in\C$,
  $\cC^\infty$ manifolds $X,\,Y$, and a homogeneous canonical relation
  $C$ from $\dot T^*Y$ to $T^*X$, $I_\cl^\mu(X,Y;C)$ denotes the
  corresponding class of classical Fourier integral operators $A\colon
  \cC_{\textup{c}}^\infty(Y)\to \cC^\infty(X)$. (Especially,
  $\Psi_\cl^m(X) = I_\cl^m(X,X;\Delta_{\dot T^*X})$, where
  $\Delta_{\dot T^*X}\subset \dot T^*X\times\dot T^*X$ is the
  diagonal.) We have a principal symbol $\sigma_\psi^\mu(A) \in
  S^{(\bar \mu)}(C;L)$ of a certain homogeneity $\bar \mu$, where $L$
  is some line bundle over~$C$.  (Both $\bar \mu$ and $L$ are of no
  further concern to us.) For $P\in \Psi_{\textup{cl}}^m(X)$ properly
  supported, $A\in I_\cl^\mu(X,Y;C)$, and $C\subset \dot T^*X
  \times\dot T^*Y$, we have $PA\in I_\cl^{m+\mu}(X,Y;C)$ and
  $\sigma_\psi^{m+\mu} (PA) = \sigma_\psi^m(P)\bigr|_C \,
  \sigma_\psi^\mu(A)$, where $\sigma_\psi^m(P)$ is lifted to a
  function on $\dot T^*X \times \dot T^*Y$ via the projection onto the
  left factor.

\item \textbf{(Transport operators)} \ In the previous item, suppose
  that $\sigma_\psi^m(P)$ vanishes on the projection of $C\subset \dot
  T^*X\times \dot T^*Y$ into $\dot T^*X$. Then $PA\in
  I_\cl^{m+\mu-1}(X,Y;C)$ and $\sigma_\psi^{m+\mu-1} (PA) = T_C^P
  \sigma_\psi^\mu(A)$, where the so-called transport operator
  $T_C^P\in \operatorname{Diff}^1(C;L)$ is a first-order differential
  operator acting on sections of the bundle $L\to C$. $T_C^P$ differs
  from the Lie derivative $\left(1/\ii\right)\mathcal
  L_{\sigma_\psi^m(P)}$ by an operator of order $0$ and is homogeneous
  of degree $m-1$, i.e., $T_C^P$ sends $S^{(\nu)}(C;L)$ into
  $S^{(\nu+m-1)}(C;L)$ for each $\nu\in\C$.
\end{itemize}


\section{Verification of results}


\subsection{Proof of Theorem~\ref{main}}

We need both existence and uniqueness.


\subsubsection*{Uniqueness}

Let $u\in \bE^0$ be a solution to $\bL u=0$. Then, by induction,
$u\in\bE^j$ for all~$j$. Indeed, suppose that $u\in \bE^j$ for some
$j\in\N_0$. Then $\bT^j \bs^j(u) = \widetilde\bs^j (\bL u)=0$. Hence,
$\bs^j(u)=0$ and $u\in \bE^{j+1}$.

Consequently, $u\in \bigcap_j \bE^j= \bE^\infty$. As $\bL\colon
\bE^\infty \to \widetilde\bE^\infty$ is bijective, we find that $u=0$
as required.


\subsubsection*{Existence}

We proceed in three steps.

\paragraph{Step 1}
Inductively, we construct $u_j\in \bE_j$ such that
\[
  g_{j+1}= \bL(u_0+u_1+\dots+u_j)-f\in \widetilde\bE^{j+1}, \quad
  j\in\N_0.
\]
Indeed, we first pick $u_0\in \bE^0$ with $\bT^0 \bs^0(u_0) =
\widetilde\bs^{\,0}(f)$. Then $g_1 =\bL u_0 - f \in \widetilde\bE^0$
and $\tilde\bs^0(g_1) = \bT^0 \bs^{\,0}(u_0) -
\widetilde\bs^{\,0}(f)=0$, hence $g_1\in \widetilde\bE^1$.

Now suppose that $u_0,\dots,u_j$ for some $j\in\N_0$ have already been
constructed. We pick $u_{j+1}\in \bE^{j+1}$ with $\bT^{j+1}\bs^{j+1}
(u_{j+1}) = -\,\widetilde\bs^{\,j+1}(g_{j+1})$. Then
\[
  g_{j+2} = \bL(u_0+\dots+u_j + u_{j+1}) - f = g_{j+1} + \bL u_{j+1}
  \in \widetilde\bE^{j+1}
\]
and
\[
  \widetilde\bs^{\,j+1}(g_{j+2}) = \widetilde\bs^{\,j+1}(g_{j+1}) +
  \bT^{j+1} \bs^{j+1}(u_{j+1})=0,
\]
hence $g_{j+2}\in \widetilde\bE^{j+2}$.


\paragraph{Step 2}
Asymptotic completeness yields a $u'\in \bE^0$ such that $u' -
\sum_{j<J}u_j\in \bE^J$ for all $J\in \N_0$. Then
\[
  \bL u' - f = \bL\left(u' - \sum\nolimits_{j<J}u_j\right) + g_J \in
  \widetilde\bE^J, \quad J\in\N_0,
\]
hence $g' =\bL u'-f \in \widetilde\bE^\infty$.


\paragraph{Step 3}
Eventually, let $v\in \bE^\infty$ be the unique solution to $\bL v =
-\,g'$. Then $u=u'+v$ is the desired solution to Eq.~\eqref{equ}.

\smallskip

This finishes the proof. \qed


\subsection{Implications of the assumptions}

For the sake of completeness, we mention two consequences of the
assumptions.

\begin{lem}
Suppose\/ \textup{(I) through (IV)} hold. Then $\widetilde\bE^0$ is
asymptotically complete for the filtration $\{\widetilde\bE^j\}$.
\end{lem}
\begin{proof}
Let $\{f_j\}\subset \widetilde \bE^0$ be a sequence with $f_j \in
\widetilde \bE^j$ for all $j$. From Theorem~\ref{main} we infer that
there are (uniquely determined) $u_j\in \bE^j$ such that $\bL u_j =
f_j$. Let $u\in \bE^0$ satisfy \eqref{nun} and $f= Lu \in
\widetilde\bE^0$. Then, for all $J\in\N_0$, $ u -
\sum\nolimits_{j<J}u_j\in \bE^J$ and
\[
  f -\sum\nolimits_{j<J} f_j = \bL \left( u - \sum\nolimits_{j<J}u_j\right)
  \in \widetilde\bE^J.
\]
Hence, $f\sim \sum f_j$ in $\widetilde\bE^0$. 
\end{proof}

Often one works in local coordinates, where one has symbolic
components to all orders. On an abstract level, this corresponds to
(the existence of and) fixing a splitting of the first row of \eqref{diagr}. So,
let us additionally assume that there are linear continous maps
$\bt^j\colon \bF^j \to \bE^j$ such that $\bs^j \circ \bt^j =
\textup{id}_{\,\bF^j}$. Then one can write
\[
  \bE^0 \cong \bE^j \oplus \bF^0 \oplus \dotsc \oplus \bF^{j-1}
\]
(algebraically and topologically),
where one passes to the next level using $\bE^j \cong \bE^{j+1} \oplus
\bF^j$ through the identification $u\overset{\cong}{\longmapsto}
\left(u-\bt^j \bs^j u,\bs^j u\right)$.

\begin{lem}
If $\bt^j$ is a splitting of the first row of \eqref{diagr}, then so
is $\bL \circ \bt^j \circ (\bT^j)^{-1}$ for the second row.
\end{lem}
\begin{proof}
We have $\widetilde\bs^{\,j} \bL \bt^j(\bT^j)^{-1} = \bT^j \bs^j
\bt^j(\bT^j)^{-1} = \textup{id}_{\,\widetilde \bF^j}$.
\end{proof}


\section{Various examples}\label{sec3}

We discuss a few examples to demonstrate the practicality of the scheme.


\subsection{Parametrices for elliptic equations}

To illustrate the elliptic case, we recall the standard result about
the existence of a parametrix to an elliptic operator. See
e.g.~\cite{Hoe1985a,Tay1981}. Notice that in the elliptic case
conditions (I) through (III) hold, but condition (IV) is usually violated.

\begin{thm}
Let $X$ be $\cC^\infty$ closed manifold and $P\in \Psi_\cl^m(X)$ be an
elliptic classical pseudodifferential operator of order $m\in\C$. Then
there is a pseudodifferential operator $Q\in \Psi_\cl^{-m}(X)$ such
that $PQ-I,\, QP-I\in \Psi^{-\infty}(X)$. $Q$ is unique modulo
$\Psi^{-\infty}(X)$.
\end{thm}
\begin{proof}
To illustrate the applicability of the abstract scheme, we show how to
construct a right parametrix. A left parametrix is constructed in the
same manner, the proof is then concluded in the usual way.

We identify all the players in the scheme: \ $\bL$ is multiplication
by $P$ from the left (and then $u=Q$, $f=I$), $\bE^j =
\Psi_\cl^{-m-j}(X)$, $\widetilde\bE^j = \Psi_\cl^{-j}(X)$, $\bF^j =
S^{(-m-j)}(\dot T^*X)$, $\widetilde\bF^j = S^{(-j)}(\dot T^*X)$,
$\bs^j=\sigma_\psi^{-m-j}$, $\widetilde\bs^{\,j}=\sigma_\psi^{-j}$,
and $\bT^j$ is multiplication by $\sigma_\psi^m(P)$.

It is well-known that (I) through (III) hold in this situation. Hence,
there is a $Q\in \Psi_\cl^{-m}(X)$ such that $PQ-I\in
\Psi^{-\infty}(X)$.
\end{proof}

\begin{rem}
The Neumann series argument usually invoked in the parametrix
construction is not required here.
\end{rem}


\subsection{A class of ordinary differential operators}

We construct a specific fundamental system for a class of ordinary
differential operators $L$ on the half-line
$\overline\R_+=[0,\infty)$. The result presented here, when it
  applies, is more precise than known results, see
  e.g.~\cite{Fed1993,Was1985,Yag1997}. Beyond that, the proof is
  considerably shorter.

Let $\ls>-1$. The ordinary differential operators $L$ under
consideration are of the form
\[
  L = D_t^m + \sum_{r=1}^m a_r(t) D_t^{m-r},
\]
where $a_r\in S_{\cl,\ls+1}^{r\ls}(\overline\R_+)$ for $1\leq r\leq
m$. Here, $b\in S_{\cl,\ls+1}^\mu(\overline\R_+)$ for $\mu\in\C$ means
that $b\in\cC^\infty(\overline\R_+)$ and there are coefficients
$b_j\in \C$ such that, for all $J,\,k\in\N_0$,
\[
  \left| \partial_t^k\left( b(t) - \chi(t) \sum\nolimits_{j<J} b_j\,
  t^{\mu-j(\ls+1)}\right) \right| \lesssim \langle t \rangle^{\Re \mu
    -J(\ls+1) - k}.
\]
Here, $\chi\in \cC^\infty(\overline\R_+)$ is an excision function,
i.e., $\chi(t)=0$ for $t\leq1$ and $\chi(t)=1$ for $t\geq2$. Writing
$a_r(t)= a_{r0}t^{r\ls} +a_{r1}t^{r\ls-(\ls+1)} +
O(t^{r\ls-2(\ls+1)})$ as $t\to\infty$, we introduce
\begin{equation}\label{odop2}
  l_0(\tau) = \tau^m + \sum_{r=1}^m a_{r0}\tau^{m-r}, \quad
  l_1(\tau) = \sum_{r=1}^m a_{r1}\tau^{m-r}.
\end{equation}

\begin{prop}
Let $\mu\in\R$ be a simple root of the polynomial $l_0$ and
\[
   \Delta = \Delta(\mu) = -\,\frac{ \ls \mu\,l_0''(\mu)/2 +
    \ii\,l_1(\mu)}{l_0'(\mu)}\in \C. 
\]
Then the equation $Lu=0$ has a unique solution of the form
\[
  u(t)=\ee^{\ii \mu\,\frac{t^{\ls+1}}{\ls+1}} c(t),
\]
where $c\in S_{\cl,\ls+1}^\Delta(\overline\R_+)$ and $c(t) = t^\Delta
+ O(t^{\Delta-(\ls+1)})$ as $t\to\infty$.
\end{prop}
\begin{proof}
For $\delta\in\C$, we set
\[
  \mathcal Z^\delta = \bigl\{\ee^{\ii \mu\,\frac{t^{\ls+1}}{\ls+1}}
  c(t)\bigm| c\in S_{\cl,\ls+1}^\delta(\overline\R_+)\bigr\}.
\]
We further set $\Gamma_0^\delta u = c_0$, $\Gamma_1^\delta u = c_1$
for $u =\ee^{\ii \mu\,\frac{t^{\ls+1}}{\ls+1}} c(t)\in \mathcal
Z^\delta$ if $c(t) = c_0 t^\delta + c_1 t^{\delta- (\ls+1)} +
O(t^{\delta - 2(\ls+1)})$ as $t\to\infty$. Notice that
$\Gamma_0^\delta u = \Gamma_1^{\delta + (\ls+1)}u$ if $u\in
\mathcal Z^\delta\subset \mathcal Z^{\delta+(\ls+1)}$.

A direct computation yields that $L\colon \mathcal Z^\delta \to
\mathcal Z^{\delta+m\ls}$, $\Gamma_0^{\delta+m\ls}(Lu) = l_0(\mu)\,
\Gamma_0^\delta u=0$, and
\[
  \Gamma_1^{\delta+m\ls}(Lu) = l_0(\mu)\, \Gamma_1^\delta u - \ii
  \left[ \delta\, l_0'(\mu) + \ls \mu \, l_0''(\mu)/2 + \ii\,
    l_1(\mu)\right] \Gamma_0^\delta u.
\]
So, $L\colon \mathcal Z^\delta \to \mathcal
Z^{\delta+m\ls-(\ls+1)}$ in view of $l_0(\mu)=0$ and
\[
  \Gamma_0^{\delta+m\ls-(\ls+1)}(Lu) = - \ii \left[ \delta\, l_0'(\mu)
    + \ls \mu \, l_0''(\mu)/2 + \ii\, l_1(\mu)\right] \Gamma_0^\delta
  u.
\]  
Note that $\Gamma_0^{\Delta+m\ls-(\ls+1)}(Lu)=0$ for all $u\in\mathcal
Z^\Delta$ by construction.

\medskip

We now apply the abstract scheme with $\bL=L$, 
\[
  \bE^j = \mathcal Z^{\Delta - j(\ls+1)}, \quad \widetilde\bE^j = \mathcal
  Z^{\Delta + m\ls - (j+1)(\ls+1)}, \quad \bF^j = \widetilde\bF^j =\C,
\]
$\bs^j= \Gamma_0^{\Delta-j(\ls+1)}$, $\widetilde\bs^{\,j} =
\Gamma_0^{\Delta+m\ls-(j+1)(\ls+1)}$, and $\bT^j$ is multiplication by
$\ii j(\ls+1)$. Then $\bE^\infty = \widetilde\bE^\infty
=\cS(\overline\R_+)$ in view of $\mu\in\R$. We observe that properties
(I) through (IV) are fulfilled for $j\geq1$.

Accordingly, we choose a $u_0 \in \mathcal Z^\Delta$ with
$\Gamma_0^\Delta u=1$. There is a unique solution $v\in \mathcal
Z^{\Delta-(\ls+1)}$ to the equation $L(u_0+v)=0$. (To see this, write
$Lv=-\,Lu_0 \in \mathcal Z^{\Delta+m\ls-2(\ls+1)}$.) Then $u=u_0+v$ is
the sought solution.
\end{proof}

\begin{cor}
Suppose that the polynomial $l_0$ has $m$ simple real roots $\mu_1<
\dotsc < \mu_m$,
\[
  l_0(\tau) = \prod_{h=1}^m (\tau- \mu_h).
\]
Then the ordinary differential operator $L$ possesses a fundamental
system
\[
  \ee^{\ii \mu_1\,\frac{t^{\ls+1}}{\ls+1}} c_1(t), \dotsc, \ee^{\ii
    \mu_m\,\frac{t^{\ls+1}}{\ls+1}} c_m(t),
\]
where $c_h \in S_{\cl,\ls+1}^{\Delta_h}(\overline\R_+)$ with $\Delta_h
= \Delta(\mu_h)$ for $1\leq h\leq m$.
\end{cor}


\subsection{The strictly hyperbolic Cauchy problem}\label{Sec32}

The following example is briefly touched upon. This example is taken from
\cite[Sec.~5]{Dui1996} and revised to fit into the abstract
scheme. More mathematical detail will be provided for the example
studied in Section~\ref{prop_conormal}, which is similar in
spirit. For the present example, the assumptions needed for the scheme
to work follow from the theory of Fourier integral operators. See
e.g.~\cite{Dui1996,Hoe1985b} as well as the notational paragraph at
the end of Section~\ref{Sec1}.

\medskip

Let $X$ be $\cC^\infty$ manifold, $P\in \Psi_\cl^m(X)$, $Q_r\in
\Psi_\cl^{m_r}(X)$ for $1\leq r\leq \mu$, and $Y\subset X$ be a closed
$\cC^\infty$ hypersurface. We assume the following:
\begin{itemize}
\item The principal symbol $p\in S^{(m)}(\dot T^* X)$ of $P$ is
  real-valued and characteristic curves of $P$ meet $Y$ transversally.
\item The algebraic problem $p(y,\xi)=0$ and $\xi\bigr|_{T_yY} = \eta$
  has, for any $(y,\eta)\in \dot T^* Y$, exactly $\mu$ solutions
  $\xi_1(y,\eta),\dots,\xi_\mu(y,\eta)\in T_y^* X$.
\item The matrix $\bigl( q_r(y,\xi_j(y,\eta))\bigr)_{j,r=1,\dots,\mu}$
  is non-singular for all $(y,\eta)\in \dot T^* Y$. Here, $q_r\in
  S^{(m_r)}(\dot T^* X)$ is the principal symbol of $Q_r$.
\end{itemize}
Let $\gamma_Y\colon \cC^\infty(X)\to \cC^\infty(Y)$, $v\mapsto
v\bigr|_Y$ be restriction to $Y$.

\begin{thm}
Under the assumptions above, there is a neighborhood $X_0\subseteq X$
of $Y$ in which a parametrix $E$ to the problem
\[
\left\{
\begin{aligned}
  Pu &= 0 && \textup{in $X_0$,} \\
  \gamma_Y Q_r u&= g_r && \textup{on $Y$,} & 1\leq r\leq \mu,
\end{aligned}
\right.
\]
exists. This parametrix $E\colon \mathscr E'(Y;\C^\mu)\to \mathscr
D'(X_0)$ is of the form
\[
  \bigl(g_1,\dotsc,g_\mu\bigr) \mapsto \sum\nolimits_{r=1}^\mu E_r
  g_r,
\]
where $E_r \in I_\cl^{-m_r-1/4}(X_0,Y;C)$ for $1\leq r\leq \mu$ with
the canonical relation $C$ being the flow-out from
$\bigl\{(x,\xi,y,\eta)\in \dot T^* X_0\times \dot T^*Y\bigm| x=y,\,
\xi\bigr|_{T_yY}=\eta,\, p(x,\xi)=0 \bigr\}$ under the Hamiltonian
vector field $H_p$.
\end{thm}
\begin{proof}
First observe that $C\subset (T^*X_0 \times T^* Y)\setminus0$ is an
embedded submanifold if the neighborhood $X_0$ of $Y$ is chosen small
enough. (It has actually $\mu$ components corresponding to the $\mu$
characteristic roots $\xi_j$.) So, the class $I_\cl^\nu(X_0,Y;C)$ for
$\nu\in\C$ of Fourier integral operators is well-defined.

Then we set up the scheme as follows: $\bL$ is the linear operator
\[
  \bigl\{E_s\bigr\}_{s=1}^\mu \mapsto
  \biggl(\bigl\{PE_s\bigr\}_{s=1}^\mu , \bigl\{\gamma_Y Q_r
  E_s\bigr\}_{r,s=1}^\mu \biggr).
\]
Moreover,
\begin{align*}
  \bE^j&= \bigoplus\nolimits_{s=1}^\mu I_\cl^{-m_s-j-1/4}(X_0,Y;C),
  \\ \widetilde\bE^j &= \bigoplus\nolimits_{s=1}^\mu
  I_\cl^{m-m_s-j-5/4}(X_0,Y;C) \oplus \bigoplus\nolimits_{r,s=1}^\mu
  \Psi_\cl^{m_r-m_s-j}(Y), \\ \bF^j &= \bigoplus\nolimits_{s=1}^\mu
  S^{(-\bar{m}_s-j-1/4)}(C;L), \\ \widetilde \bF^j &=
  \bigoplus\nolimits_{s=1}^\mu S^{(m-\bar{m}_s-j-5/4)}(C;L)\oplus
  \bigoplus\nolimits_{r,s=1}^\mu S^{(m_r-m_s-j)}(\dot T^* Y),
\end{align*}
$\bs^j$ and $\widetilde\bs^{\,j}$ are composed of the respective
principal symbol maps, and the transport operators are given as
\[
  \bT^j = \Bigl(\{T_C^P\}_{1\leq s\leq\mu},H^j\Bigr)
\]
where the linear operator $H^j\colon \bF^j\to
\bigoplus\nolimits_{r,s=1}^\mu S^{(m_r-m_s-j)}(\dot T^* Y)$ is
composition with the principal symbol of $\gamma_Y\in
I_\cl^{1/4}(Y,X_0;R)$ (here $R=(\dot N_{Y\times X_0}\Delta_Y)'$)
followed by pointwise multiplication from the left by the matrix
$\bigl( q_r(y,\xi_j(y,\eta))\bigr)_{j,r=1,\dots,\mu}$. This provides
initial conditions for solving the transport equations.

By construction, conditions (I) through (III) are fulfilled.
\end{proof}

The (microlocal) constructions above can be globalized upon making
additional global assumptions on the characteristic flow. Moreover,
for the genuine strictly hyperbolic Cauchy problem and $P$ being a
differential operator, $\cC^\infty$ well-posedness turns the
parametrix $E$ into a fundamental solution. For details, see again
\cite{Dui1996,Hoe1985b}.


\subsection{Propagation of conormality}\label{prop_conormal}

Propagation of conormality (see Appendix~\ref{conormal_distributions})
is much related to the well-posedness of the hyperbolic Cauchy
problem. It is a more precise statement than the better known
statement about the mere propagation of singularities (which, in turn, is a
microlocalized version of the basic energy inequalities), insofar as
the kind of singularities being propagated gets specified.

We consider the strictly hyperbolic Cauchy problem. Coordinate
invariance being known, we can work in local coordinates $(t,x)\in
(0,T)\times\R^n$. So, let
\[
  L = D_t^m + \sum_{\substack{j+|\alpha|\leq m,\\j<m}}
  a_{j\alpha}(t,x) D_t^j D_x^\alpha
\]
with coefficients $a_{j\alpha}\in
\cC_{\textup{b}}^\infty([0,T]\times\R^n)$. We assume $L$ to be
strictly hyperbolic, i.e.,
\[
  \sigma_\psi^m(L)(t,x,\tau,\xi) =
  \prod_{h=1}^m (\tau - \mu_h(t,x,\xi))
\]
for real-valued $\mu_h \in S^{(1)}([0,T]\times\R^n\times
\dot\R^n)$, where, in addition,
\[
  \left|\mu_h(t,x,\xi) - \mu_{h'}(t,x,\xi)\right|\gtrsim |\xi|, \quad
  h\neq h'.
\]  

We are interested in the inhomogeneous Cauchy problem
\begin{equation}\label{cauchy2}
\left\{ \enspace
\begin{aligned}
  & L u = f(t,x), && (t,x)\in (0,T)\times\R^n, \\
  & D_t^p u\bigr|_{t=0} = g_p(x) , && 0\leq p\leq m-1.
\end{aligned}
\right.
\end{equation}

Let $\Sigma\subset \R^n$ by a $\cC^\infty$ hypersurface, say of
bounded geometry. Let $T>0$ be small enough so that we have $m$
characteristic hypersurfaces $\Sigma_1,\dots, \Sigma_m \subset
[0,T]\times \R^n$ of $L$ emanating from $\{0\}\times \Sigma$, with
$\Sigma_h$ belonging to the $h$\/th characteristic root~$\mu_h$.

\begin{thm}
Let $\mu\in\C$, $f= \sum_{h=1}^m f_h$, where $f_h\in
I_\cl^{m+\mu-1}([0,T]\times\R^n, \Sigma_h)$ for $1\leq h\leq m$, and
$g_p\in I_\cl^{\mu+p+1/4}(\R^n,\Sigma)$ for $0\leq p\leq m-1$. Then
the unique solution $u$ to \eqref{cauchy2} is of the form $u=
\sum_{h=1}^m u_h$, where $u_h\in I_\cl^\mu([0,T]\times\R^n,\Sigma_h)$
for $1\leq h\leq m$.
\end{thm}
\begin{proof}
The bundles $\Omega_{\dot N^* \Sigma_h}^{1,\textup{v}}$ (see
Appendix~\ref{conormal_distributions}) are trivialized as we work in
local coordinates. Hence, $\sigma_\psi^\nu(v)\in S^{(\bar\nu)}(\dot
N^*\Sigma_h)$ for $v\in I_\cl^\nu([0,T]\times \R^n,\Sigma_h)$, where
now $\bar\nu = \nu +(n-1)/4$. Similarly for $g\in
I_\cl^{\nu+1/4}(\R^n;\Sigma)$ and $\sigma_\psi^\nu(g)\in
S^{(\bar\nu)}(\dot N^*\Sigma)$.

We apply the abstract scheme employing the following setup: \ $\bL u =
\left(Lu,u\bigr|_{t=0},D_tu\bigr|_{t=0},\dotsc,\right.$
$\left.D_t^{m-1}u\bigr|_{t=0}\right)$,
\[
\begin{aligned}
  \bE^j &= \biggl\{\sum\nolimits_{h=1}^m u_h\biggm| \text{$u_h\in
    I_\cl^{\mu-j}([0,T]\times\R^n,\Sigma_h)$ for $1\leq h\leq m$}
  \biggr\}, \\ \widetilde\bE^j &= \biggl\{\sum\nolimits_{h=1}^m
  f_h\biggm| \text{$f_h\in
    I_\cl^{m+\mu-j-1}([0,T]\times\R^n,\Sigma_h)$ for $1\leq h\leq m$}
  \biggr\} \oplus \bigoplus_{p=0}^{m-1}
  I_\cl^{\mu+p-j+1/4}(\R^n,\Sigma).
\end{aligned}
\]
To justify this, note that $\sigma_\psi^{m+\mu-j}(Lu)=0$ for $u\in
I_\cl^{\mu-j}([0,T]\times\R^n,\Sigma_h)$ due to the fact that
$\sigma_\psi^m(L)\bigr|_{\dot N^*\Sigma_h}=0$. We further set
\[
  \bF^j = \bigoplus_{h=1}^m S^{\bar\mu-j}(\dot N^*\Sigma_h), \quad
  \widetilde\bF^j = \bigoplus_{h=1}^m S^{m+\bar\mu-j-1}(\dot
  N^*\Sigma_h) \oplus \bigoplus_{p=0}^{m-1} S^{(\bar\mu+p-j)}(\dot
  N^*\Sigma)
\]
and
\begin{align*}
  \bs^j(u) &= \Bigl( \{\sigma_\psi^{\mu-j}(u_h)\}_{1\leq h\leq
    m}\Bigr), \\
  \widetilde\bs^{\,j}\bigl(f, \{g_p\}_{0\leq p\leq m-1}\bigr) &=
  \Bigl(\{\sigma_\psi^{m+\mu-j-1}(f_h)\}_{1\leq h\leq m},
  \{\sigma_\psi^{\mu+p-j+1/4}(g_p)\}_{0\leq p\leq m-1} \Bigr)
\end{align*}
writing $u=\sum_h u_h$ and $f=\sum_h f_h$ as above. The transport
operators are
\[
  \bT^j = \Bigl( \{T_{\Sigma_h}^p\}_{1\leq h\leq m},H^j\Bigr)
\]
where $H^j \colon \bF^j \to \bigoplus_{p=0}^{m-1} S^{(\bar\mu+p-j)}(\dot
N^*\Sigma)$ is (up to a multiplicative factor of $(2\pi)^{-1}$) the
linear operator of restriction to $\{t=0\}$ followed by pointwise
multiplication from the left by the matrix
$\bigl(\mu_h(0,x,\xi)^p\bigr)_{\substack{1\leq h\leq m,\\0\leq p\leq
    m-1}}$ for $(x,\xi)\in \dot N^*\Sigma$. To see this note that
$\sigma_\psi^{\nu+p}(D_t^p v) = \tau^p \sigma_\psi^\nu(v)$ for $v\in
I_\cl^\nu([0,T]\times\R^n;\Sigma_h)$ and $\tau= \mu_h (0,x,\xi)$ for
$(0,x,\tau,\xi)\in \dot N^*\Sigma_h$. Further note that the matrix
$\bigl(\mu_h(0,x,\xi)^p\bigr)_{\substack{1\leq h\leq m,\\0\leq p\leq
    m-1}}$ is non-singular as a Vandermonde matrix, considering that
the characteristic roots are pairwise distinct.

Hence, the $\bT^j$ are indeed bijective (as implied by the method of
characteristics) and conditions (I) through (IV) of the abstract
scheme are fulfilled. Notice that condition (IV) is nothing else than
the $\cC^\infty$ well-posedness of the Cauchy problem \eqref{cauchy2}.
\end{proof}


\subsection{Green's functions}

It is a well-known fact that the future light cone is the singular
support of the forward Green's function (henceforth called Green's
function) of the wave operator and that, in odd space dimensions, the
Green's function is smooth from both sides up to the future light
cone.\footnote{Indeed, it is zero outside the future light cone by
finite propagation speed.} This property holds in greater generality
as we are going to show now. As we shall see, only certain algebraic
properties of the amplitude functions, which arise when writing the
conormal distributions under consideration as oscillatory integrals,
are indeed responsable for that result to hold, and these algebraic
properties are propagated. General references are
\cite{Bou1979,HP1979}.

\begin{rem}
A similar discussion is possible for even space dimensions. In this
case, the result is that the Green's function of the wave operator is
zero outside the future light cone, in particular, smooth from the
outside of the future light cone up the future light cone, and it is
what one calls diffuse from the inside.
\end{rem}

We consider a second-order strictly hyperbolic differential operator
$P$ with coefficients in $\cC_{\textup{b}}^\infty([0,T]\times\R^n)$,
which we write as
\begin{equation}\label{ppP}
  P =D_t^2 + p_1(t,x,D_x)D_t + p_2(t,x,D_x).
\end{equation}
We denote by $q_l(t,x,\xi)$ for $l=1,2$ the principal symbol of
$p_l(t,x,D_x)$. We assume that $q_1$, $q_2$ are real-valued and that
\[
  q_2(t,x,\xi) \lesssim -|\xi|^2.
\]

\begin{ex}
An example is the wave operator $\Box =D_t^2 +c^2(t,x) \Delta_x$
with variable propagation speed, where $c\in
\cC_{\textup{b}}^\infty([0,T]\times\R^n)$ and $c(t,x)\gtrsim 1$.
\end{ex}

Let $\mu^\pm\in S^{(1)}([0,T]\times\R^n\times\dot\R^n)$ with 
\[
  \mu^\pm(t,x,\xi) = -\,\frac{1}2\,q_1(t,x,\xi) \pm
  \frac12\sqrt{q_1^2(t,x,\xi) - 4 q_2(t,x,\xi)}
\]
denote the characteristic roots of $P$. Notice that $\pm \,
\mu^\pm(t,x,\xi) \gtrsim |\xi|$ and that
\begin{equation}\label{srt}
  \mu^\pm(t,x,-\xi) = -\,\mu^\mp(t,x,\xi).
\end{equation}
Further let $[0,T]\ni t \mapsto \gamma^\pm(t;x^0,\xi^0)=\left(
t,x(t),\tau(t),\xi(t)\right)$ for $(x^0,\xi^0)\in\R^n\times\dot\R^n$
denote the null bicharacteristic emanating from $\left(0,x^0,
\mu^\pm(0,x^0,\xi^0),\xi^0\right)$. It is given as the solution to the
system
\[
\left\{ \enspace
\begin{aligned}
  & x'(t) = -\,\mu_\xi^\pm(t,x(t),\xi(t)), \quad \xi'(t) =
  \mu_x^\pm(t,x(t),\xi(t)), \\ & x(0) = x^0, \quad \xi(0) = \xi^0,
\end{aligned}
\right.
\]
and $\tau(t) = \mu^\pm(t,x(t),\xi(t))$ follows. Notice that in view
of \eqref{srt} we have that
\begin{equation}\label{lup}
  x^\pm(t;x^0;\xi^0) = x^\mp(t;x^0;-\xi^0), \quad \xi^\pm(t;x^0;\xi^0)
  = -\,\xi^\mp(t;x^0;-\xi^0),
\end{equation}
whereas $\pm\tau^\pm(t;x^0,\xi^0)>0$.

We also introduce the Lagrangian manifolds with boundary
\[
  \Lambda^\pm = \{\gamma^\pm(t;0,\xi^0)\mid t\in [0,T],\,\xi^0\in
  \dot\R^n\} \subset \dot{T}^\ast([0,T]\times\R^n).
\]
Notice that $\Lambda^+\cap \Lambda^-= \emptyset$. $\Lambda^\pm$ are
parametrized by the non-degenerate phase functions
$\phi^\pm=\phi^\pm(t,x,\xi)$ which are found as solutions to the
eikonal equations
\begin{equation}\label{eikonal}
\left\{ \enspace
\begin{aligned}
  \partial_t \phi^\pm &= \mu^\pm(t,x,\nabla \phi^\pm), && (t,x)\in (0,T)\times\R^n, \\
  \phi^\pm\bigr|_{t=0} &= x\cdot\xi.
\end{aligned}
\right.
\end{equation}
Here, $\xi\in\dot\R^n$ acts as a parameter. These are Hamilton-Jacobi
equation which are solved utilizing the method of
characteristics. Notice that, in view of our assumptions,
Eq.~\eqref{eikonal} can be solved up to time $t=T$.

\begin{ex}
For the wave operator $P = D_t^2 + c^2(t)\Delta_x$ with $c$
independent of $x$, we have
\[
  \Lambda^\pm = \{(t,x,\tau,\xi)\mid x = \mp\, C(t) \xi/|\xi|, \,
  \tau= \pm\, c(t)|\xi|,\,t\in[0,T],\, \xi\in\dot\R^n\}.
\]  
and $\phi^\pm(t,x,\xi) = x\cdot\xi \pm C(t)|\xi|$, where $C(t) =
\int_0^t c(s)\,\dd s$.
\end{ex}

At last, we introduce the future light cone with vertex at the origin
$\mathscr O=(0,0)$,
\[
  \mathcal C = \{(t,x)\in [0,T]\times\R^n\mid
  \text{$x=x^\pm(t;0,\xi^0)$ for some $\pm\,\xi^0\in \dot\R^n$}\},
\]
see \eqref{lup}

\begin{lem}\label{332}
For $T>0$ small, $\mathcal C \setminus \mathscr O\subset
(0,T]\times\R^n$ is a hypersurface\footnotemark. Furthermore,
\[
  \dot{N}^*(\mathcal C \setminus \mathscr O) = (\Lambda^+ \sqcup
  \Lambda^-) \cap \{0<t\leq T\}.
\]  
For $(t,x,\tau,\xi) \in \dot{N}^*(\mathcal C \setminus \mathscr O)$,
we have that $(t,x,\tau,\xi) \in \Lambda^\pm$ if and only if $\pm
\tau>0$.
\end{lem}
\footnotetext{The fact that $\mathcal C \setminus \mathscr O$ has a
  boundary at $t=T$ plays no role here.}

In view of Lemma~\eqref{332}, we set
\[
  I_\cl^\mu([0,T]\times\R^n,\mathcal C) =
  I_\cl^\mu([0,T]\times\R^n,\Lambda^+) +
  I_\cl^\mu([0,T]\times\R^n,\Lambda^-)
\]
for $\mu\in\C$ and say that distributions in
$I_\cl^\mu([0,T]\times\R^n,\mathcal C)$ are classically conormal, of
order $\mu$, with respect to $\mathcal C$. Notice that
$I_\cl^\mu([0,T]\times\R^n,\Lambda^+) \cap
I_\cl^\mu([0,T]\times\R^n,\Lambda^-) =
\cC_{\textup{b}}^\infty([0,T]\times\R^n)$.

\begin{rem}
It holds 
\[
  I_\cl^\mu([0,T]\times\R^n,\mathcal C) \subset \cC^\infty([0,T];\cS'(\R^n)).
\]
More precisely, for $u \in I_\cl^\mu([0,T]\times\R^n,\mathcal C)$, we
have that $u\bigr|_{\{t\}\times\R^n}\in I_\cl^{\mu+1/4}(\R^n,\mathcal
C \cap (\{t\}\times \R^n))$ for $0<t\leq T$, while
$u\bigr|_{\{0\}\times\R^n}\in I_\cl^{\mu+1/4}(\R^n,\{0\})$.
\end{rem}

We now consider the inhomogeneous Cauchy problem
\begin{equation}\label{tat}
\left\{ \enspace
\begin{aligned}
  & Pu = f(t,x), \qquad (t,x)\in (0,T)\times\R^n, \\
  & u\bigr|_{t=0} = g_0(x), \quad D_t u\bigr|_{t=0} = g_1(x).
\end{aligned}    
\right.
\end{equation}

We first show that the solution $u$ is conormal with respect to $\mathcal C$
under the condition that the initial data $(g_0,g_1)$ is conormal with
respect to the origin and that the right-hand side $f$ is also conormal with
respect to $\mathcal C$.

\begin{thm}\label{hasd}
Let $\mu\in \C$, $g_0\in I_\cl^{\mu+1/4}(\R^n,\{0\})$, $g_1\in
I_\cl^{\mu+5/4}(\R^n,\{0\})$, and $f\in
I_\cl^{\mu+1}([0,T]\times\R^n,\mathcal C)$. Then the unique solution
$u$ to\/ \textup{Eq.~\eqref{tat}} belongs to
$I_\cl^\mu([0,T]\times\R^n,\mathcal C)$.
\end{thm}

\begin{proof}
We apply the abstract scheme with $\bL u = \bigl(Pu, u\bigr|_{t=0},
D_t u\bigr|_{t=0}\bigr)$ and
\begin{align*}
  \bE_j &= I_\cl^{\mu-j}([0,T]\times\R^n,\mathcal C), \\
  \widetilde \bE_j &= I_\cl^{\mu-j+1}([0,T]\times\R^n,\mathcal C) \oplus
  I_\cl^{\mu-j+1/4}(\R^n,\{0\}) \oplus
  I_\cl^{\mu-j+5/4}(\R^n,\{0\}), \\
  \bF_j &= S^{(\bar\mu-j)}(\Lambda^+\sqcup \Lambda^-), \\
  \widetilde \bF_j &= S^{(\bar\mu-j+1)}(\Lambda^+\sqcup
  \Lambda^-) \oplus S^{(\bar\mu-j)}(\dot\R^n) \oplus
  S^{(\bar\mu-j+1)}(\dot\R^n),
\end{align*}
where $\bar \mu= \mu - \frac{n-1}4$. The maps $\bs^j\colon \bE^j \to
\bF^j$ and $\widetilde\bs^j\colon \widetilde \bE^j\to \bF^j$ are the
principal symbol maps, while the operator $\bT^j \colon \bF^j \to
\widetilde\bF^j$ is given as $\bT^j=\left(T_{\Lambda^+}^P\oplus
T_{\Lambda^-}^P,\gamma_0,\gamma_1\right)$, where
\begin{equation}\label{dert}
\left\{ \enspace
\begin{aligned}
&\gamma_0\colon \bigl(a_{(\bar\mu-j)}^+, a_{(\bar\mu-j)}^-\bigr)
  \mapsto a_{(\bar\mu-j)}^+\bigr|_{t=0} +
  a_{(\bar\mu-j)}^-\bigr|_{t=0}, \\
  &\gamma_1\colon \bigl(a_{(\bar\mu-j)}^+, a_{(\bar\mu-j)}^-\bigr)
  \mapsto \mu^+ a_{(\bar\mu-j)}^+\bigr|_{t=0} + \mu^-
  a_{(\bar\mu-j)}^-\bigr|_{t=0}.
\end{aligned}
\right.
\end{equation}
As before, we readily establish that properties (I) through (IV) hold.
\end{proof}

\begin{rem}
The solution $u$ to \eqref{tat} is of the form
\begin{equation}\label{tpp0}
  u(t,x) = \int_{\R^n} \left( \ee^{\ii \phi^+(t,x,\xi)} a^+(t,\xi) +
  \ee^{\ii \phi^-(t,x,\xi)} a^-(t,\xi) \right)\dbar \xi
  \mod{\cC_{\textup{b}}^\infty}
\end{equation}
for suitable $a^\pm \in S_\cl^{\bar\mu}([0,T]\times\R^n)$.
\end{rem}

As announced at the beginning of this section, we now strengthen the
statement of Theorem~\ref{hasd}. To this end, we need to discuss the
transmission property, see Appendix~\ref{trprop}. From now on, the space
dimension $n$ is odd.

\begin{prop}
Let $\bar\mu=\mu-(n-1)/4\in\Z$ and $n$ be odd. Then $u=u(t,x)$ as
given in \eqref{tpp0} possesses the two-sided transmission property
with respect to $\mathcal C$ if and only if
\begin{equation}\label{tpp1}
  a_{(\bar\mu-j)}^-(t,-\xi) = (-1)^{\bar\mu-j} a_{(\bar\mu-j)}^+(x,\xi),
  \quad j\in \N_0.
\end{equation}
\end{prop}
\begin{proof}
This follows by writing, for $0<t\leq T$, 
\[
  u(t,x) = \int_{-\infty}^\infty \ee^{\ii(t-\psi(x))\tau} b(x,\tau)\,\dbar
  \tau \mod \cC^\infty
\]
for a suitable\footnote{$b=b(x,\tau)$ grows like $|x|^{-(n-1)/2}$ as
$x\to 0$, but this is irrelevant here.} $b\in
S_\cl^{\bar\mu+(n-1)/2}((\R^n\setminus\{0\})\times\R)$, where
$t-\psi(x)=0$ is a defining equation for $\mathcal C\setminus
\mathscr O$, and, using the stationary phase method, by expressing
$b_{(\bar\mu-j +(n-1)/2)}(x,\tau)$ for $j\in\N_0$ in terms of
$a_{(\bar\mu-k)}^+(t,\xi)$ for $k\leq j$ when $\tau>0$ and in terms of
$a_{(\bar\mu-k)}^-(t,\xi)$ for $k\leq j$ when $\tau<0$. The condition
\[
  b_{(\bar\mu-j +(n-1)/2)}(x,-\tau)= (-1)^{(\bar\mu-j +(n-1)/2)}
  b_{(\bar\mu-j +(n-1)/2)}(x,-\tau), \quad j\in\N_0,
\]
then translates into \eqref{tpp1}. See also \cite{Bou1979,HP1979}.
\end{proof}

For the proof of the next theorem, we need to work out suitable
subspaces of the function spaces $\bE^j$, $\widetilde \bE^j$, $\bF^j$,
and $\widetilde\bF^j$ used in the proof of Theorem~\ref{hasd}. For
$\bar\mu\in\Z$ and $n$ odd, we define the space
$I_{\cl,\tr}^\mu([0,T]\times\R^n,\mathcal C)$ to consist of all $u\in
I_\cl^\mu([0,T]\times\R^n,\mathcal C)$ that possess the two-sided
transmission property with respect to the future light cone $\mathcal
C$. Further we introduce
\[
  S_\tr^{(\bar\mu)}(\Lambda^+\sqcup \Lambda^-) =
  \{(a_{(\bar\mu)}^+,a_{(\bar\mu)}^-)\in S^{(\bar\mu)}(\Lambda^+\sqcup
  \Lambda^-) \mid a_{(\bar\mu)}^-(t,-\xi) = (-1)^{\bar\mu}
  a_{(\bar\mu)}^+(t,\xi)\}.
\]
We also need to impose certain symmetry conditions on the initial data
$(g_0,g_1)$. For $\bar\mu\in\Z$, let
\[ 
  \tilde S^{(\bar\mu)}(\dot\R^n) = \left\{b_{(\bar\mu)}\in
  S^{(\bar\mu)}(\dot\R^n) \mid b_{(\bar\mu)}(-\xi) =
  (-1)^{\bar\mu}b_{(\bar\mu)}(\xi) \right\}
\]
and let $\tilde I_\cl^{\mu+1/4}(\R^n,\{0\})$ consist of all $g\in
I_\cl^{\mu+1/4}(\R^n,\{0\})$ such that the Fourier transform
$\widehat{g} \in S_\cl^{\bar\mu}(\R^n)$ admits an asymptotic expansion
$\widehat g(\xi) \sim \sum_{j\geq 0} \chi(\xi) b_{(\bar\mu-j)}(\xi)$
with $b_{(\bar\mu-j)}\in \tilde S^{(\bar\mu-j)}(\dot\R^n)$ for all
$j$.

\begin{lem}
Let $\bar\mu\in\Z$ and $n$ be odd. Then, for $u \in
I_{\cl,\tr}^\mu([0,T]\times\R^n,\mathcal C)$, 
\[
  u\bigr|_{t=0} \in \tilde I_\cl^{\mu+1/4}(\R^n,\{0\}), \quad 
  D_t u\bigr|_{t=0} \in \tilde I_\cl^{\mu+5/4}(\R^n,\{0\}).
\]
\end{lem}
\begin{proof}
A direct verification that employs \eqref{dert}.
\end{proof}

\begin{thm}\label{additional}
Let $\bar \mu\in \Z$ and $n$ be odd. Suppose that $g_0\in\tilde
I_\cl^{\mu+1/4}(\R^n,\{0\})$, $g_1\in\tilde
I_\cl^{\mu+5/4}(\R^n,\{0\})$, and $f\in
I_{\cl,\tr}^{\mu+1}([0,T]\times\R^n,\mathcal C)$. Then the solution
$u$ as described in the previous theorem belongs to the space
$I_{\cl,\tr}^{\,\mu}([0,T]\times\R^n,\mathcal C)$, i.e., it possesses the
two-sided transmission property with respect to the future light cone
$\mathcal C$.
\end{thm}
\begin{proof}
We change the setup in the proof of the previous theorem as follows:
\begin{align*}
  \bE_j &= I_{\cl,\tr}^{\mu-j}([0,T]\times\R^n,\mathcal C), \\
  \widetilde \bE_j &= I_{\cl,\tr}^{\mu-j+1}([0,T]\times\R^n,\mathcal C) \oplus
  \tilde I_\cl^{\mu+1/4}(\R^n,\{0\}) \oplus
  \tilde I_\cl^{\mu+5/4}(\R^n,\{0\}), \\
  \bF_j &= S_\tr^{(\bar\mu-j)}(\Lambda^+\sqcup \Lambda^-), \\
  \widetilde \bF_j &= S_\tr^{(\bar\mu-j+1)}(\Lambda^+\sqcup
  \Lambda^-) \oplus \tilde S^{(\bar\mu-j)}(\dot\R^n) \oplus
  \tilde S^{(\bar\mu-j+1)}(\dot\R^n).
\end{align*}
The operators $\bs^j$, $\widetilde\bs^j$, and $\bT^j$ are as before,
but acting on the indicated subspaces. Now, given $b_{(\bar\mu-j)} \in
\tilde S^{(\bar\mu-j)}(\dot\R^n)$, $ c_{(\bar\mu-j+1)} \in \tilde
S^{(\bar\mu-j+1)}(\dot\R^n)$ and solving the linear system
\[
  \gamma_0 \bigl( a_{(\bar\mu-j)}^+\bigr|_{t=0},
  a_{(\bar\mu-j)}^-\bigr|_{t=0} \bigr) = b_{(\bar\mu-j)}, \quad
  \gamma_1 \bigl( a_{(\bar\mu-j)}^+\bigr|_{t=0},
  a_{(\bar\mu-j)}^-\bigr|_{t=0} \bigr) = c_{(\bar\mu-j+1)}
\]
for $a_{(\bar\mu-j)}^+\bigr|_{t=0}$, $a_{(\bar\mu-j)}^-\bigr|_{t=0}$,
we find that $a_{(\bar\mu-j)}^+\bigr|_{t=0},
a_{(\bar\mu-j)}^-\bigr|_{t=0} \in \tilde S^{(\bar\mu-j)}(\dot\R^n)$
and
\begin{equation}\label{ssz}
  a_{(\bar\mu-j)}^-\bigr|_{t=0}(-\xi) = (-1)^{\bar\mu-j
  }a_{(\bar\mu-j)}^+\bigr|_{t=0}(\xi).
\end{equation}
Consequently, in view of the symmetry properties of the transport
operators $T_{\Lambda^+}^P$, $T_{\Lambda^-}^P$, the relation in
\eqref{ssz}~continues to hold also for $0<t\leq T$ and we readily
obtain that the operator
\[
  \bT^j \colon S_\tr^{(\bar\mu-j)}(\Lambda^+\sqcup\Lambda^-) \to
  S_\tr^{(\bar\mu-j+1)}(\Lambda^+\sqcup\Lambda^-) \oplus \tilde
  S^{(\bar\mu-j)}(\dot\R^n) \oplus \tilde
  S^{(\bar\mu-j+1)}(\dot\R^n)
\]
is an isomorphism. Again, properties (I) through (IV) are satisfied.
\end{proof}

\begin{ex}
Let $n$ be odd. Then the Green's function of the operator $P$ in
\eqref{ppP} is found when $g_0 = 0$, $g_1=\ii \delta(x) \in \tilde
I_\cl^{n/4}(\R^n,\{0\})$, and $f=0$. We have $\bar\mu=-1$, and all the
conditions of Theorem~\ref{additional} are fulfilled. Consequently, we
recover the known result that the forward Green's function of $P$ is
smooth from both sides up the future light cone.
\end{ex}


\appendix


\section{}


\subsection{Asymptotic completeness}

Here, we report an abstract scheme for establishing asymptotic
completeness. It is in the spirit of this article. Details can be
found in \cite[Prop.~1.1.17]{Sch1998}.

As before, let $\bigl\{\bE^j\bigr\}_{j\in\N_0}$ be a sequence of
Fr\'echet spaces such that
\[
  \bE^0 \supseteq \bE^1 \supseteq \bE^2 \supseteq \dotsc \supseteq
  \bE^\infty,
\]
where $\bE^\infty =\bigcap_j \bE^j$. Suppose that there is family
$\{\bbchi(c)\mid c\geq1\}$ of linear operators such that
\begin{itemize}
\item $\bbchi(c) \in \bigcap_{j\in\N_0} \mathcal L(\bE^j,\bE^j)$,
\item $1 - \bbchi(c) \colon \bE^0 \to \bE^\infty$,
\item $\bbchi(c) u \to 0$ as $c\to\infty$ in $\bE^j$ for all $u\in
  \bE^{j+1}$.
\end{itemize}

\begin{thm}
Under the conditions above, for each sequence $\{u_j\}\subset \bE^0$
with $u_j\in\bE^j$ for all~$j$, there is a sequence $\{c_j\}\subset
[1,\infty)$ \/\textup{(}with $c_j\to\infty$ as $j\to \infty$
  sufficiently fast depending on $\{u_j\}$\textup{)} such that, for
  any $J\in\N_0$, the series
\[
  \textup{$\sum\nolimits_{j\geq J} \bbchi(c_j)u_j$ converges in
    $\bE^J$.}
\]
\end{thm}

Especially, setting $u =\sum\nolimits_{j\geq 0} \bbchi(c_j)u_j$, we
have that $u\sim \sum\nolimits_{j\geq 0} u_j$ in $\bE^0$. Indeed, for
$J\in\N_0$,
\[
  u -\sum\nolimits_{j<J}u_j =
  -\,\sum\nolimits_{j<J}\left(1-\bbchi(c_j)\right) u_j +
  \sum\nolimits_{j\geq J} \bbchi(c_j)u_j \in \bE^J.
\]


\subsection{Conormal distributions}\label{conormal_distributions}

We remind the reader of the concept of a conormal distribution and
list some of the basic properties of those distributions. For details,
see e.g.~\cite[Sec.~18.2]{Hoe1985a}.

Let $X$ be a $\cC^\infty$ manifold, of dimension $n$, and $Y\subset X$
be a closed $\cC^\infty$ submanifold, of codimension $k\leq n$. A
distribution $u\in\cD'(X)$ is said to be conormal with respect to $Y$,
of order $\nu\in\R$, if $T_1\dotsc T_l u\in
B_{2,\infty}^{-\nu-n/4}(X)$ for any number $l$ of vector fields $T_j$
on $X$ tangent to $Y$ (where $B_{p,q}^s(X)$ is to designate the local
Besov spaces on $X$). This class of conormal distributions is denoted
by $I^\nu(X,Y)$. We have that $\operatorname{WF}(u)\subseteq \dot N^*
Y$ for $u\in I^\nu(X,Y)$, where $\dot N^* Y$ is the conormal bundle of
$Y$ in $X$ with the zero section removed. In particular,
$u\bigr|_{X\setminus Y} \in \cC^\infty(X\setminus Y)$.

There is an alternative description of conormal distributions through
oscillatory integrals. Specifically, in local coordinates $x=(x',x'')
\in \R^k \times \R^{n-k}$ such that $Y=\{x'=0\}$,
\[
  u(x) = \int_{\R^k} \ee^{\ii x'\xi'} a(x'',\xi') \,\dbar \xi'
  \mod{\cC^\infty(\R^n)},
\]
where $a\in S^{\nu+(n-2k)/4}(\R^{n-k}\times\R^k)$ and $\dbar\xi' =
(2\pi)^{-(n+2k)/4}\,\dd\xi'$. The subclass $I_\cl^\mu(X,Y)\subset
I^{\Re\mu}(X,Y)$ for $\mu\in\C$ of classical conormal distributions
results upon requiring that $a\in
S_\cl^{\mu+(n-2k)/4}(\R^{n-k}\times\R^k)$. In this case, $a(x'',\xi')
= \chi(\xi') a_{(\bar\mu-k)}(x'',\xi') + a_r(x'',\xi')$, where
$a_{(\bar\mu-k)}\in S^{(\bar\mu-k)}(\R^{n-k}\times\dot\R^k)$, $\bar\mu
= \mu +(n+2k)/4$, $\chi\in \cC^\infty(\R^k)$, $\chi(\xi')=0$ for
$|\xi'|\leq 1$, $\chi(\xi')=1$ for $|\xi'|\geq2$, and $a_r\in
S_\cl^{\bar\mu-k-1}(\R^{n-k}\times\R^j)$. Then
\[
  \sigma_\psi^\mu(u) = a_{(\bar\mu-k)}(x'',\xi')\left|\dd \xi'\right|
\]
is the local coordinate expression for the principal symbol
$\sigma_\psi^\mu(u) \in S^{(\bar\mu)} (\dot N^* Y;\Omega_{\dot N^*
  Y}^{1,\textup{v}})$ with $\Omega_{\dot N^* Y}^{1,\textup{v}} =
\Omega_{\dot N^* Y}^{1/2} \otimes \bigl(\pi\bigr|_{\dot N^* Y}\bigr)^*
\Omega_X^{-1/2}$ being the vertical $1$-density bundle over $\dot N^*
Y$. Here, $\Omega_Z^\alpha$ is the $\alpha$-density bundle over a
manifold $Z$ and $\pi\colon \dot T^* X\to X$ denotes the canonical
projection.

\medskip

In Section~\ref{prop_conormal}, we make use of the following
properties:

\smallskip
\noindent
(a) \ \textbf{(Principal symbol map)} \ The principal symbol map fits
into a short exact sequence
\[
  0 \longrightarrow I_\cl^{\mu-1}(X,Y) \longrightarrow I_\cl^\mu(X,Y)
  \overset{\sigma_\psi^\mu}{\longrightarrow} S^{(\bar\mu)} (\dot N^*
  Y;\Omega_{\dot N^* Y}^{1,\textup{v}}) \longrightarrow 0,
\]
which splits. As a consequence, $I_\cl^\mu(X,Y)$ is a nuclear
Fr\'echet space equipped with the projective limit topology of
$\prod_{j<J}S^{(\bar\mu-j)}(\dot N^* Y;\Omega_{\dot N^*
  Y}^{1,\textup{v}}) \times I^{\Re \mu-J}(X,Y)$ as $J\to\infty$, with
each factor carrying its natural Fr\'echet topology.

\smallskip
\noindent
(b) \ \textbf{(Application of pseudodifferential operators)} \ For
$P\in \Psi_\cl^m(X)$ and $u\in I_\cl^\mu(X,Y)$, we have $Pu \in
I_\cl^{m+\mu}(X,Y)$ and $\sigma_\psi^{m+\mu}(Pu) =
\sigma_\psi^m(P)\bigr|_{\dot N^* Y} \, \sigma_\psi^\mu(u)$. If $Y$ is
characteristic for $P$ (i.e., $\sigma_\psi^m(P)\bigr|_{\dot N^*
  Y}=0$), then $Pu \in I_\cl^{m+\mu-1}(X,Y)$ and
$\sigma_\psi^{m+\mu-1}(Pu) = T_Y^P \sigma_\psi^mu(u)$, where $T_Y^P
\in \operatorname{Diff}^1(\dot N^* Y;\Omega_{\dot N^*
  Y}^{1,\textup{v}})$ is a first-order differential operator acting on
sections of the bundle $\Omega_{\dot N^* Y}^{1,\textup{v}}$ (similar
to the transport operators $T_C^P$ introduced at the end of
Section~\ref{Sec1}).

\smallskip
\noindent
(c) \ \textbf{(Restriction to submanifolds)} \ Let $u\in
I_\cl^\mu(X,Y)$ and $Z\subset X$ be a closed $\cC^\infty$ submanifold
that meets $Y$ transversally. Then $u\bigr|_Z \in
I_\cl^{\mu+l/4}(Z,Y\cap Z)$, where $l$ is the codimension of $Z$ in
$X$, and $\sigma_\psi^{\mu+l/4}\bigl(u\bigr|_Z\bigr) =(2\pi)^{-l/4}
\sigma_\psi^\mu(u)\bigr|_{\dot N_Z^*(Y\cap Z)}$.

The restriction to $\dot N_Z^*(Y\cap Z)$ is to be understand as
follows: \ $T_pY+ T_pZ =T_p X$ for $p\in Y\cap Z$ implies that
$T_pX/T_pY \cong T_pZ/T_p(Y\cap Z)$ and, therefore, $N_X
Y\bigr|_{Y\cap Z}\cong N_Z(Y\cap Z)$
canonically for the normal bundles. Passing to the dual bundles,
$N_Z^*(Y\cap Z) \cong N_X^* Y\bigr|_{Y\cap Z} \hookrightarrow N_X^*Y$.

\smallskip
\noindent
(d) \ \textbf{(Asymptotic completeness)} \ Let $u_j \in
I_\cl^{\mu-j}(X,Y)$ for $j\in\N_0$. Then there exists a $u\in
I_\cl^\mu(X,Y)$ with the property that $u-\sum_{j<J} u_j\in
I_\cl^{\mu-J}(X,Y)$ for all $J\in\N_0$. This $u$ is uniquely
determined modulo $\cC^\infty(X)$.


\subsection{The transmission property}\label{trprop}

Let $X$ be a $\cC^\infty$ manifold and $Y\subset X$ be a closed
$\cC^\infty$ hypersurface. We discuss here the two-sided transmission
property for distributions $u\in \cD'(X)$ with respect to $Y$. For a
more detailed discussion, see \cite{Bou1979, HP1979, Hoe1985a}.

\begin{defn}
$u\in I_\cl^\mu(X,Y)$ is said to have the two-sided transmission
  property with respect to $Y$ if $u\bigr|_{X\setminus Y}\in
  \cC^\infty(X\setminus Y)$ extends smoothly (locally from both sides)
  to $Y$.
\end{defn}

A necessary condition is that $\bar \mu = \mu + \frac{n}4-\frac12 \in
\Z$, where $n= \dim X$.

The two-sided transmission property is a local property. In local
coordinates $x= (x',x_n)\in\R^{n-1}\times\R$, let $Y=\{x_n=0\}$. Then
$u\in I_\cl^\mu(X,Y)$ if and only if
\begin{equation}\label{tr1}
   u(x) = \int_{-\infty}^\infty \ee^{\ii x_n\xi} a(x',\xi)\,\dbar \xi
   \bmod{\cC_{\textup{b}}^\infty},
\end{equation}
where $a\in S_\cl^{\bar \mu} (\R^{n-1}\times\R)$. Let
\begin{equation}\label{tr2}
  a(x',\xi) \sim \sum_{k\geq0} \chi(\xi) a_{(\bar\mu-k)}(x',\xi),
\end{equation}
where $ a_{(\bar\mu-k)} \in S^{(\bar\mu-k)}(\R^{n-1}\times\dot\R)$.

\begin{prop}
For $\bar\mu\in\Z$, the distribution $u\in I_\cl^\mu(X,Y)$ given in
\eqref{tr1} with \eqref{tr2} possesses the two-sided transmission
property if and only if
\[
  a_{(\bar\mu-k)}(x',-\xi) = (-1)^{\bar\mu-k} a_{(\bar\mu-k)}(x',\xi),
  \quad k\in\N_0.
\]
\end{prop}

\begin{rem}
Let $\bar\mu\in\Z$ and $u\in I_\cl^\mu(X,Y)$ have the two-sided
transmission property. Then, locally,
\[
   u = u\bigr|_{\{x_n\neq0\}} + \sum_{0\leq j\leq \bar\mu} c_j(x')
   \otimes \delta^{(j)}(x_n),
\]
where $c_j\in \cC_{\textup{b}}^\infty(\R^{n-1})$. In particular, $u$
is a regular distribution if $\bar\mu<0$.
\end{rem}

\begin{rem}
For general $\mu\in \C$, one similarly defines a one-sided
transmission property (locally with respect to either side of
$Y$). All three conditions are equivalent if $\bar\mu\in\Z$.
\end{rem}


\nocite{*}
\bibliographystyle{abbrv}
\bibliography{abstract_scheme.bib}


\end{document}